 \documentclass[authoryear,preprint,12pt]{elsarticle}

\journal{Journal of Mathematical Analysis and Applications}
\usepackage{subcaption}
\usepackage{geometry}

\usepackage{amssymb}
\usepackage{tikz}
\usepackage{caption}
\usetikzlibrary{calc, patterns,arrows.meta}
\usepackage{mathtools}
\usepackage{amsthm}

\usepackage[capitalize,noabbrev]{cleveref}
\usepackage{amsmath}

\newcommand{\R}{\mathbb{R}}

\newcommand{\N}{\mathbb{N}}

\newcommand{\bF}{\mathcal{F}}
\newcommand{\sm}{\mathcal{S}}
\newcommand{\dist}{\mathrm{Dist}}

\newcommand{\afix}{\mathrm{AsFix}}
\newcommand{\CDPS}{\textbf{CIPS}}
\newcommand{\fix}{\mathrm{Fix}}

\newcommand{\taup}{\tau_1}
\newcommand{\taum}{\tau_2}
\newcommand{\bx}{\boldsymbol{x}}
\newcommand{\bn}{\boldsymbol{n}}

\newcommand{\by}{\boldsymbol{y}}

\newcommand{\bv}{\boldsymbol{v}}
\newcommand{\bu}{\boldsymbol{u}}

\newcommand{\br}{\boldsymbol{r}}
\newcommand{\bp}{\boldsymbol{p}}

\newcommand{\bb}{\boldsymbol{b}}

\newcommand{\bW}{\boldsymbol{W}}

\newcommand{\bD}{\boldsymbol{D}}

\newcommand{\bS}{\boldsymbol{S}}

\newcommand{\be}{\boldsymbol{e}}
\newcommand{\bJ}{\boldsymbol{J}}

\newcommand{\bom}{\boldsymbol{w}}
\newcommand{\bmu}{\boldsymbol{\mu}}

\newcommand{\bI}{\boldsymbol{I}}
\newcommand{\bi}{\boldsymbol{1}}

\newcommand{\bX}{\boldsymbol{X}}

\newcommand{\Jac}{\boldsymbol{J}}

\newcommand{\bxi}{\boldsymbol{\xi}}

\newcommand{\conv}{\mathrm{cv}}

\newtheorem{theorem}{Theorem}[section]

\newtheorem{remark}[theorem]{Remark}
\newtheorem{lemma}[theorem]{Lemma}
\newtheorem{corollary}[theorem]{Corollary}
\theoremstyle{definition}
\newtheorem{definition}[theorem]{Definition}

\begin{document}

\begin{frontmatter}

\title{On Unstable Fixed Points in Modern Continuous Hopfield Networks}
	\author{
 	Hans-Peter Beise\\
 	}
\address{Department of Computer Science\\
Trier University of Applied Sciences\\
Schneidershof, 54293 Trier\\
Beise@hochschule-trier.de}

\begin{abstract}
The recently introduced continuous Hopfield network \cite{ramsauer2020hopfield} exhibits large memorization capabilities, which manifest as attractive fixed points of its update rule—a differentiable function consisting of two linear mappings composed with the scaled softmax function. The authors of \cite{ramsauer2020hopfield} provide proofs for the existence and approximate position of such attractive fixed points. For the softmax function alone, the fixed point structure has been fully characterized in earlier work \cite{tivno2009bifurcation}, from which it turns out that for sufficiently large scaling factors there are exponentially more unstable fixed points than attractive ones. In this work, we complement the findings in \cite{ramsauer2020hopfield} by showing that, under natural geometric conditions on the vectors defining the continuous Hopfield network, unstable fixed points must occur, analogous to the findings in \cite{tivno2009bifurcation}. Our results show that, under these geometric conditions, continuous Hopfield networks necessarily admit additional unstable fixed points associated with higher-dimensional faces of the pattern polytope.
\end{abstract}

\end{frontmatter}
\footnotetext{\copyright\ 2025. This manuscript version is made available 
under the CC-BY-NC-ND~4.0 license. 
\texttt{https://creativecommons.org/licenses/by-nc-nd/4.0/}}

\section{Introduction}\label{sec:Introduction}
Hopfield networks, introduced in \cite{hopfield1982neural}, are a classical class of neural networks that implement \textit{associative memory} via stable, attractive fixed points. Recent research has been dedicated to extending the basic idea of Hopfield networks. A unifying framework and overview of some recent contributions is presented in \cite{krotov2021large}. Of particular interest for this work is the so-called \emph{modern continuous Hopfield network}, introduced in \cite{ramsauer2020hopfield}. The update rule, referred to as the \textbf{continuous Hopfield function} in the sequel, has the form
\begin{equation}\label{networkFunHopf}
	f(\bx) = \bW \,\sm_\beta(\bW^T\bx),
\end{equation}
with \(\bW \in \R^{d \times n}\) consisting of columns \(\bom_1, \ldots, \bom_n \in \R^d\), which represent the patterns to be stored. The \textbf{scaled softmax function} is defined by
\begin{equation}\label{def_softmax}
    \sm_\beta(\bx) = \left( \frac{\exp(\beta \,\bx^{(1)})}{Z}, \ldots, \frac{\exp(\beta \,\bx^{(n)})}{Z} \right)^T, \quad Z := \sum_{j=1}^n \exp(\beta\, \bx^{(j)}).
\end{equation}

The function \(f\) should memorize the given patterns as attractive, stable fixed points. That is, for a given input $\bxi \in \R^d$, the iterative application of $f$ on $\bxi$ converges to a fixed point near some $\bom_j$. The continuous Hopfield function \eqref{networkFunHopf} exhibits close connections to the scaled dot-product attention function used in transformers \cite{vaswani2017attention}, the neural network architecture underlying large language models. This attention function writes as
\begin{equation}\label{scaledAttention}
	A(\bX) = \bW_3\bX \,\sm_\beta(\,\bX^T\bW_2^T \bW_1\bX),
\end{equation}
where $\bX$ is an input matrix and the $\bW_j$, $j = 1,2,3$, are trainable parameter matrices and $\sm_\beta$ is applied column-wise. Based on this observation, different transformer-like architectural blocks designed to implement explicit memorization mechanisms are proposed in \cite{ramsauer2020hopfield}. The utility of this approach is demonstrated in \cite{widrich2020modern}, for instance. We will not explicitly focus on this attention function in the sequel. 

For sufficiently large \(\beta\), it has been shown in \cite{ramsauer2020hopfield} that \(f\) possesses attractive fixed points in small neighborhoods of the vectors \(\bom_j\). If $\bom_1,\ldots,\bom_n$, the columns of $\bW$, are pairwise distinct and lie on a sphere (cf.~\cite[Equation (308)]{ramsauer2020hopfield}), then for sufficiently large $\beta$ there are attractive fixed points of $f$ in small neighborhoods of each $\bom_j$. Indeed, their inequality \cite[Equation (311)]{ramsauer2020hopfield} always holds if $\beta$ is sufficiently large, and the results in the sequel then yield these fixed points. In this work, we will assume that the $\bom_j$ are pairwise distinct unit-length vectors. The main property derived from this condition is that each $\bom_j$ defines a corner of the convex hull of $\bom_1,\ldots,\bom_n$, and our analysis can be adapted to hold if only this latter condition is satisfied.

The fixed point results in \cite{ramsauer2020hopfield} are derived by means of the Banach fixed point theorem, which always yields attractive fixed points. The complete fixed point structure, particularly the existence of unstable fixed points, is not covered.

The scaled softmax function \(\sm_\beta\) itself has a well-structured fixed point distribution, as analyzed in detail in \cite{tivno2009bifurcation}. In that work, a complete understanding of the evolution of fixed points as $\beta$ changes is given. It should be noted that the analysis in \cite{tivno2009bifurcation} is formulated in terms of $1/\beta$, which is often referred to as the temperature. A crucial observation underlying this analysis is that, for a fixed point of $\sm_\beta$,
\begin{equation}\label{3fp_eq1}
    \frac{\exp(\beta\, \bx^{(j)})}{Z} = \bx^{(j)}
\end{equation}
for all $j$, where $Z := \sum_{j=1}^n \exp(\beta \bx^{(j)})$. Due to the strict convexity of the exponential function, this equation can have at most two solutions, which implies that the components of $\bx$ decompose into at most two classes within which the components take equal values. Following this, it is observed that the fixed points can only lie on line segments in the standard simplex $\Delta^{n-1}$ passing through its arithmetic mean $1/n\,\bi$, with $\bi := (1,\ldots,1)^T$, connecting two faces of dimension $k$ and $l$ with $k + l = n - 1$. Indeed, it is shown in \cite{tivno2009bifurcation} that for $\beta > n$, there are two fixed points on each such line in addition to the fixed point $1/n\,\bi$. Among these $2^n - 1$ fixed points, only the one on the line to each vertex of $\Delta^{n-1}$ is attractive for such a $\beta$. As $\beta$ increases, these fixed points move towards the respective faces of $\Delta^{n-1}$. We provide a more rigorous formulation of the exact results of \cite{tivno2009bifurcation} in \ref{sec:append}.

The aforementioned results suggest that a similar behavior may occur in modern continuous Hopfield networks $f$ in \eqref{networkFunHopf}, and motivate our investigation of their fixed point structure of continuous Hopfield functions beyond attractive fixed points. We show that, under suitable geometric conditions on the stored patterns, modern continuous Hopfield networks possess additional unstable fixed points. These additional fixed points appear near certain faces of the convex polytope spanned by the stored patterns. Figure~\ref{fig:netdynamics} illustrates this by means of a simple numerical example.

\section{Main Result}\label{sec_mainResult}
In this section, we prepare our main result and give an informal version.
We begin by introducing the notations employed throughout. Vectors and matrices are denoted by boldface letters, and $\bx^{(j)}$ denotes the $j$-th component of a vector $\bx$. By $\be_1,\ldots,\be_n$ we mean the standard unit vectors in $\R^n$. The Euclidean norm is denoted by $\Vert \bx\Vert$. By $\Delta^{n-1}$ we denote the standard simplex 
\[
\Delta^{n-1} := \left\{ \bx \in [0,\infty)^n : \sum_{j=1}^n \bx^{(j)} = 1 \right\}.
\]
The vector $\bi \in \R^n$ is defined by $\bi := (1,\ldots,1)^T$. The composition of functions is denoted by $\circ$. For a finite set $M$, we write $|M|$ for its cardinality. The convex hull of a set $M \subset \R^d$ is denoted by $\conv M$ and its boundary is denoted by $\partial M$. For $n \in \N$, we set $[n] := \{1, \ldots, n\}$. For maxima over index sets we use the extended-real convention $\max \emptyset := -\infty$ and $\min\emptyset := +\infty$.

Additional notation is introduced as needed in subsequent sections.

Let \( f: \R^d \to \R^d \) be a map. A point \( \bx^* \in \R^d \) is called a \textbf{fixed point} of \( f \) if \( f(\bx^*) = \bx^* \). The set of all fixed points is denoted by \( \fix(f) \).
A fixed point \( \bx^* \) is called \textbf{attractive} if there exists an open neighborhood \( \mathcal{U} \ni \bx^* \) such that for all \( \bx \in \mathcal{U} \), the iterates \( f^{\circ k}(\bx) \to \bx^* \) as \( k \to \infty \). The set of all $\bx$ that converge to $\bx^*$ in this way is known as the \textbf{basin of attraction} of $\bx^*$.
A fixed point is further called \textbf{asymptotically stable} if it is attractive and if for every open neighborhood \( \mathcal{U} \) of \( \bx^* \), there exists an open neighborhood \( \mathcal{O} \) of \( \bx^* \) such that
\begin{equation}\label{defstatt}
    \{ f^{\circ k}(\bx) :\, \bx \in \mathcal{O},\, k \in \N \} \,\subset\, \mathcal{U}.
\end{equation}
For continuously differentiable maps on an open domain, as considered here, asymptotic stability follows if the spectral radius $\rho$ of the Jacobian $\Jac$  satisfies \( \rho(\Jac(\bx^*)) < 1 \). The set of all asymptotically stable fixed points is denoted by \( \afix(f) \). Conversely, if the spectral radius is greater than one, the fixed point is said to be \textbf{unstable}. That is, there exist directions in which nearby points are repelled, and hence $\bx^*$ is not locally attractive. However, it is possible in principle that orbits initially repelled may later return and converge along other directions. For the case \( \rho(\Jac(\bx^*)) = 1 \), all behaviors are possible: the fixed point may be asymptotically stable, unstable, or stable but not attractive.
In the context of Hopfield networks, the term \emph{spurious fixed point} refers to fixed points that are not explicitly intended to be memorized pattern. In our case that is, they do not lie near any pattern vector \( \bom_1,\ldots\bom_n\).


The fixed points of the scaled softmax function $\sm_\beta$ are known to lie on line segments connecting faces of $\Delta^{n-1}$, as analyzed in \cite{tivno2009bifurcation} and briefly revisited in \ref{sec:append}. In contrast, for $f(\bx) = \bW\,\sm_\beta(\bW^T\bx)$ as in \eqref{networkFunHopf}, an explicit analytical characterization of its fixed points is generally out of reach. To prove the existence of (possibly unstable) fixed points in this case, we make use of tools from topological degree theory.
The topological mechanism underlying our main result is entailed in the following theorem on fixed points. This result can be considered as fixed point version of the Poincar\'e–Miranda theorem \cite{miranda1940osservazione, kulpa1997poincare}. We refer to Section \ref{sec_proofs} for details on the notation.  A short proof of this version following \cite{vrahatis1989short} is provided later.

\begin{theorem}[Poincar\'e--Miranda type fixed points]
\label{prop_miranda}
Let $f:P\rightarrow \R^d$ be continuous on a convex, compact polytope $P\subset \R^d$ with non-empty interior. Let $H(F)$ denote the supporting hyperplane of a facet $F\in \bF^{d-1}(P)$, where $H(F)=\{\bx\in \R^d: \bn_F^T \bx=b_F\}$ for some outward pointing normal vector $\bn_F\in \R^d$ and constant $b_F\in \R$, i.e. $P\subset\{\bx\in \R^d: \bn_F^T\,\bx \le b_F\}$. 

Assume that the facets can be partitioned into two disjoint classes $C_1\cup C_2=\bF^{d-1}(P)$ with associated orthogonal subspaces $U_1,U_2\subset\R^d$ satisfying $U_1\oplus U_2=\R^d$, such that for every $F\in C_k$, $k\in\{1,2\}$, we have $\bn_F\in U_k$, and moreover
\begin{align}
f(\bx)&\in \{\bx\in \R^d: \bn_F^T\,\bx \,\le\, b_F\} &&\text{for all }F\in C_1\text{ and all }\bx\in F,\label{mirandaCondition_01}\\
f(\bx)&\in \{\bx\in \R^d: \bn_F^T\,\bx \,\ge\, b_F\} &&\text{for all }F\in C_2\text{ and all }\bx\in F.\label{mirandaCondition_02}
\end{align}
Then $f$ has a fixed point in $P$.
\end{theorem}

We next introduce a geometric condition on subsets of the pattern vectors $\bom_1, \ldots, \bom_n$ that enables application of \cref{prop_miranda}.

\begin{definition}\label{def_CDPS}
Let $\bom_1, \ldots, \bom_n \in \R^d$. A subset $\{\bom_j : j \in J\}$ with $J \subset [n]$ is called \textbf{Convexly Inner Product Separated (\CDPS)} with respect to $J_0\subset [n]$, with $J_0\cap J=\emptyset$, if
\begin{equation*}
    \max\{\bom_j^T \bv : j \in J\} > \max\{\bom_j^T \bv : j \in J_0\}\quad\mathrm{for \, all}\ \bv \in \conv\{\bom_j : j \in J\} .
\end{equation*}
We will assume $J_0=[n]\setminus J$, if not explicitly stated otherwise, and simply write that the $\{\bom_j : j \in J\}$ are \CDPS.

\end{definition}

We now state an informal version of our main result. A rigorous version, together with proofs, is presented in the following section.

\begin{theorem}[Informal]\label{theo_hopf_miranda}
Let $f$ be the continuous Hopfield function defined in \eqref{networkFunHopf}, with pairwise distinct, unit-length vectors $\bom_1, \ldots, \bom_n \in \R^d$ as columns of $\bW$ and $J\subset [n]$.
\begin{enumerate}
    \item Suppose $\{\bom_j : j \in J\}$ is \CDPS\, and that each facet of $P := \conv\{\bom_j : j \in J\}$ corresponds to a \CDPS\, subset. Then, for sufficiently large $\beta > 0$, $f$ has a fixed point located in a neighborhood of $P$. Moreover, as $\beta \to \infty$, this fixed point converges to $P$.

    \item If, in addition, every face of $P$ corresponds to a \CDPS\, subset, then for sufficiently large $\beta$, $f$ has a fixed point associated with each such face. Among these face-associated fixed points, only the fixed points near the vertices (i.e., the $\bom_j$) are asymptotically stable.
\end{enumerate}
\end{theorem}

\begin{figure*}[h]
\vskip 0.2in
\begin{center}
		    \includegraphics[width=0.2\linewidth]{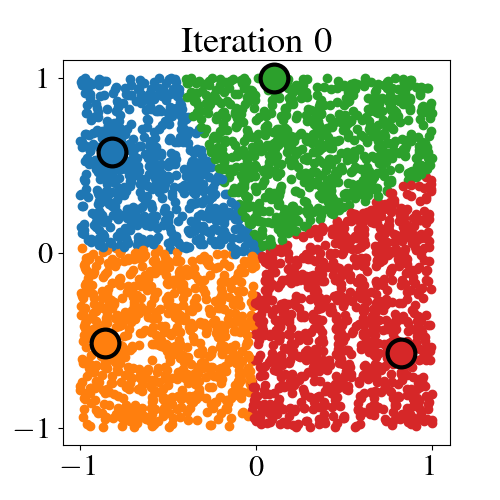}\hfill
			\includegraphics[width=0.2\linewidth]{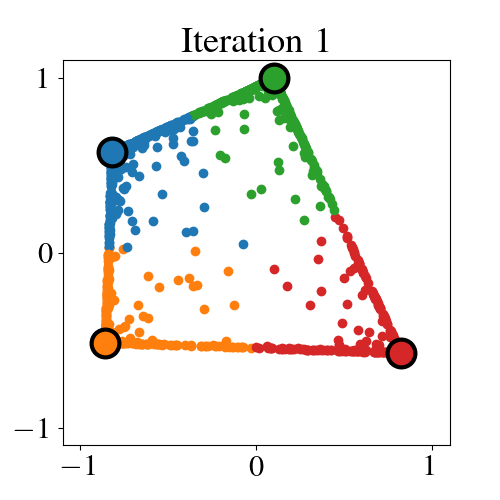}\hfill
			\includegraphics[width=0.2\linewidth]{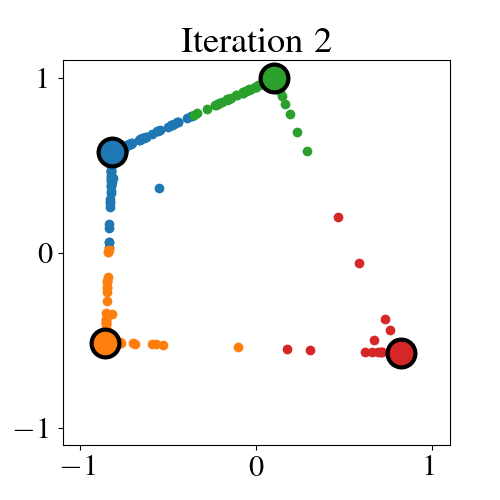}\hfill
                \includegraphics[width=0.2\linewidth]{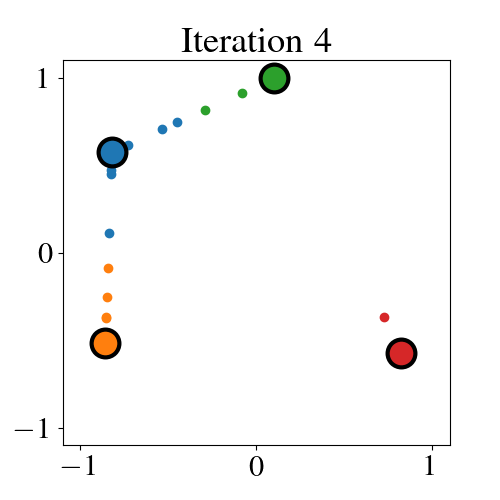}\hfill
              \includegraphics[width=0.2\linewidth]  
            {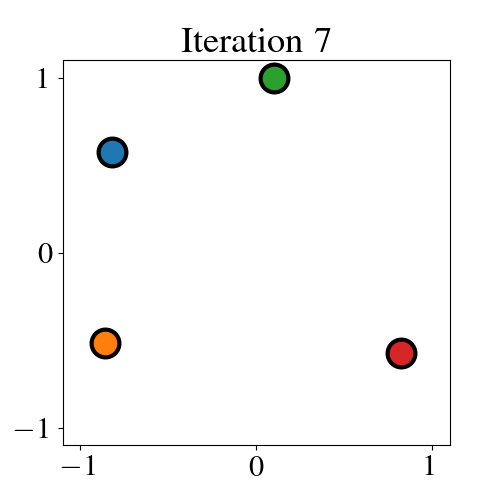}\hfill
		    \includegraphics[width=0.2\linewidth]{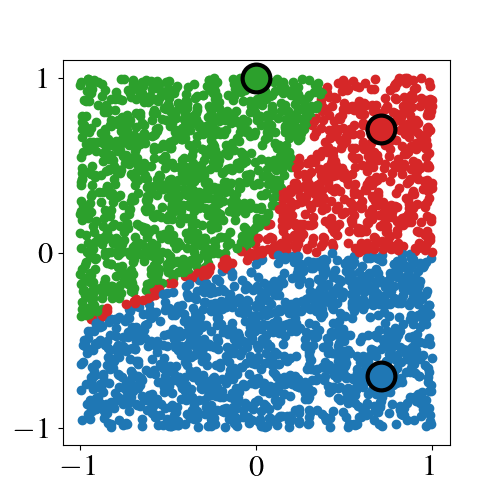}\hfill
			\includegraphics[width=0.2\linewidth]{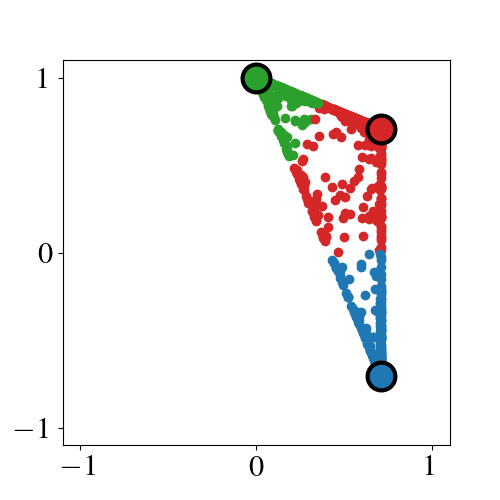}\hfill
			\includegraphics[width=0.2\linewidth]{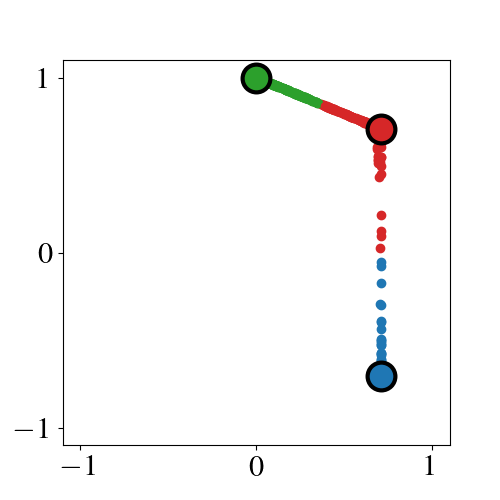}\hfill
            \includegraphics[width=0.2\linewidth]{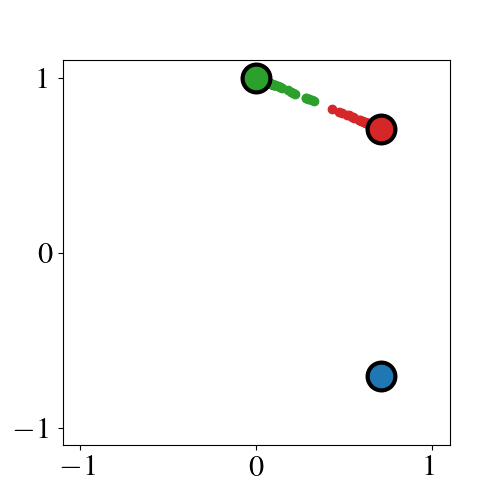}\hfill
              \includegraphics[width=0.2\linewidth]  
            {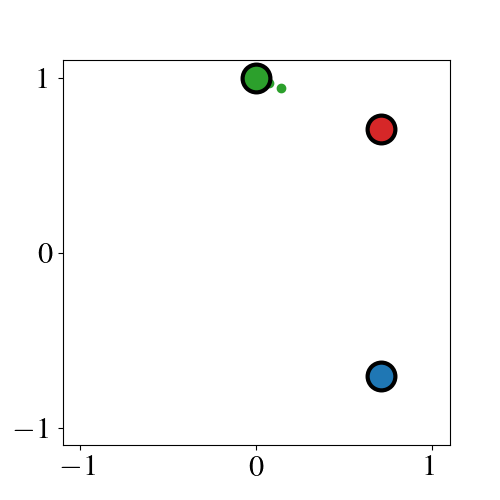}\hfill

    \caption{Dynamics of continuous Hopfield function (\ref{networkFunHopf}) for 2D data with $\beta=15$. Small dots depict the evolution of 3000 uniform random points in $[-1,1]^2$. Attractive fixed points, approximately equal to four vectors $\bom_j$ defining $f$ are represented as large, bordered dots. The evolution after 0, 1, 2, 4, and 7 iterative applications of $f$ is shown from left to right. The colors correspond to the fixed point to which the points converge.  
}
	\label{fig:netdynamics}
 \end{center}
\vskip -0.2in
\end{figure*}
Figure~\ref{fig:netdynamics} illustrates how the dynamics of $f$ can exhibit both, asymptotically stable and unstable fixed points, depending on the geometric arrangement of the $\bom_j$. In particular, the presence of unstable fixed points near 1-faces of $\conv\{\bom_1,\ldots,\bom_n\}$ becomes evident. The latter observation can be generalized to higher dimensions under certain conditions. Indeed, consider the case $\bom_j = \be_j$ for $j = 1, \ldots, n$, i.e., $\bW$ is the identity and thus $f = \sm_\beta$. The structure of $\fix(\sm_\beta)$ evolves stepwise with increasing $\beta$ as laid out in more detail in \ref{sec:append} (cf. Figures~\ref{fig:hk_01},~\ref{fig:hk_02}, and \cite[Figure 2]{tivno2009bifurcation}). That means, for fixed $n$, there exists a threshold $m(n)$ such that
$\beta < m(n)$ implies $\fix(\sm_\beta) = \afix(\sm_\beta) = \left\{\frac{1}{n} \bi\right\}$.
Whereas for $\beta > m(n)$, the number of fixed points grows, as $\beta$ increases, from $\vert \fix(\sm_\beta)\vert=2n+1$ and reaches its maximum of $2^n - 1$ fixed points once $\beta > n$. In this course, there are several 
bifurcation points, similar to $m(n)$, beyond which the number of fixed points increases, c.f \cite{tivno2009bifurcation} and \ref{sec:append}. As lower thresholds $m(n)$, we have for instance $m(3) \approx 2.75$, $m(10) \approx 4.56$, and $m(1000) \approx 10.11$ (see Table~\ref{tab:min_h_n}).
When the vectors $\bom_1, \ldots, \bom_n$ are approximately orthogonal, one expects similar behavior. However, for more distorted configurations, the situation may differ significantly.

We next formulate a lemma that quantifies how the geometry of the stored patterns influences the contraction behavior of $f$ near certain faces of the convex hull. This will be used in Section~\ref{sec_proofs} to apply \cref{theo_hopf_miranda}.

\begin{lemma}\label{lem_hopf_conv} 
Let $f$ be the continuous Hopfield function defined in (\ref{networkFunHopf}), with unit-length vectors $\bom_1,\ldots,\bom_n\in \R^d$ being the columns of $\bW$. For indices $J\subset[n]$ and some $\bv\in \R^d$
let \begin{equation*}
\taup(\bv)\, :=\, \max_{j\in J}\bom_j^T \bv, \quad \taum(\bv)\, := \,\max_{j\in [n]\setminus J}\bom_j^T \bv.
\end{equation*}
and $\delta := \taup - \taum$. Then for $\bx=\bv+\varepsilon \br$ with $\bv\in\conv\{\bom_j:\, j\in J\}$, $\br\in\R^d$ a unit-length vector, and $\varepsilon>0$, we have
\begin{equation}\label{dist_estimate}
    \dist\left(f(\bx), \conv\{\bom_j:\, j\in J\}\right)\,<\,2\,(n-\vert J\vert
)\, \exp(\beta(-\delta+2\varepsilon)).
\end{equation}
\end{lemma}

This lemma helps explain the dynamics observed in Figure~\ref{fig:netdynamics}, where orbits of points under iterated application of $f$ are seen to approach the convex hull of the pattern vectors, before converging to the attractive fixed points. The unstable fixed points near lower-dimensional faces appear at the boundaries of basins of attraction. Furthermore, the bottom row of Figure~\ref{fig:netdynamics} suggests that the existence of such unstable fixed points depends sensitively on the geometry of the pattern configuration. For instance, if three unit-length vectors $\bom_1, \bom_2, \bom_3$ are such that $\bom_1$ and $\bom_2$ are orthogonal and $\bom_3:=1/\sqrt{2}(\bom_1+\bom_2)$. Then $\{\bom_1, \bom_2\}$ is not a \CDPS\, subset of $\{\bom_1,\bom_2,\bom_3\}$. In this case, for sufficiently large $\beta$, \cref{lem_hopf_conv} implies that the vectors on $\conv\{\bom_1,\,\bom_2\}$ are either attracted by $\conv\{\bom_1,\,\bom_3\}$ or $\conv\{\bom_2,\,\bom_3\}$ and in the neighborhood of $\bom_1,\, \bom_2$ we have a unique attractive fixed points by the Banach fixed point theorem, cf. \cite[Lemma A6]{ramsauer2020hopfield}. Thus, there will be no fixed point near $\conv\{\bom_1,\,\bom_2\}$ other than these attractive fixed points.

\begin{remark}\label{rem:not_exact_convHull}
\begin{enumerate}
    \item The \CDPS\, condition is not vacuous. For instance,
    when the $\bom_j$ are obtained by applying a linear transformation
    with moderate condition number to the standard basis vectors, numerical sampling shows the condition holds for a substantial fraction of randomly sampled faces (see
    Section~\ref{sec:num_experiments} for details).
    
    \item Figure~\ref{fig:netdynamics} illustrates that, for sufficiently large values of $\beta$ the dynamics is strongly attracted toward lower-dimensional faces of the convex hull. The value $\beta=15$ used in the figure is chosen to make this behavior visually apparent. This indicates the utility of \cref{lem_hopf_conv} in our analysis. Sufficient magnitudes of $\beta$ in \cref{theo_hopf_miranda} and their relation to the scaling in transformers \cite{vaswani2017attention}, are briefly discussed in Section~\ref{sec:num_experiments}. 
    
\end{enumerate}
\end{remark}

\section{Mathematical Formulation and Proofs}
\label{sec_proofs}

In this section, we provide the rigorous version of \cref{theo_hopf_miranda} and the proofs. Before doing so, we need to introduce some additional notation.

For sets \( A, B \subset \mathbb{R}^d \), a point \( \bx \in \mathbb{R}^d \), and a scalar \( \mu \in \mathbb{R} \), we use the following standard notation:
\[
A \pm B := \{ \mathbf{x} \pm \mathbf{y} : \mathbf{x} \in A,\, \mathbf{y} \in B \}, \quad 
\dist(\mathbf{x}, A) := \inf_{\mathbf{y} \in A} \Vert \mathbf{x} - \mathbf{y} \Vert, \quad 
\mu A := \{ \mu \mathbf{x} : \mathbf{x} \in A \}.
\]

We also use basic notions from the theory of polytopes in \( \mathbb{R}^d \), i.e., bounded convex sets defined as the intersection of finitely many half-spaces \cite{grunbaum1967convex}. For example, the convex hull \( \conv\{ \boldsymbol{\omega}_1, \ldots, \boldsymbol{\omega}_k \} \) defines a polytope.
Let \( \dim(P) \) denote the dimension of the smallest affine subspace containing the polytope \( P \). For \( 0 \leq k \leq \dim(P) \), let \( \bF^k(P) \) denote the set of \( k \)-faces of \( P \), and let
\[
\partial^k P := \bigcup_{F \in \bF^k(P)} F
\]
be the union of all \( k \)-faces. The elements of \( \bF^{\dim(P) - 1}(P) \) are also called \emph{facets} of \( P \), and for simplicity we denote $\partial P:= \partial^{\dim(P) - 1}P$. Below, the polytopes are usually defined by a set of vectors $\{\bv_j :\, j \in J\}$, where $J \subset \N$ is an index set and each $\bv_j$ defines a unique vertex of $P = \conv\{\bv_j :\, j \in J\}$. For a given face $F$ of $P$, we write $I(F) \subset J$ for the indices defining $F$, that is $F = \conv\{\bv_j :\, j \in I(F)\}$.

For a bounded polytope \( P \subset \mathbb{R}^d \), let \( \mathbf{b}_P \) denote the arithmetic mean of its vertices (i.e., its 0-faces). For \( \eta \in (0,1) \) define the \textit{scaled polytope}
\begin{equation}\label{def_pmu}
    P_\eta \,:=\, \eta(P - \mathbf{b}_P) + \mathbf{b}_P.
\end{equation}

Let \( m := d - \dim(P) \) and let \( \bv_1, \ldots, \bv_m \) be an orthonormal basis for the subspace orthogonal to the affine hull of \( P \). Define the auxiliary polytope
\begin{equation}\label{def:BPperp}
  C(P^\perp) \,:=\, \left\{\sum_{j=1}^m \alpha_j\, \tilde{\bv}_j:\, \sum_{j=1}^m \alpha_j\, \leq\, 1\, \mathrm{and }\, \alpha_j\geq 0,\, \tilde{\bv}_j\in\{\pm\bv_j\}\ \mathrm{for }\, j=1,\ldots,m \right\}.  
\end{equation}

Although the definition of \( C(P^\perp) \) depends on the particular choice of the orthonormal vectors \( \bv_1, \ldots, \bv_m \), it is only important that we have a formal means of defining a blow up in all directions orthogonal to \( P \). The vectors \( \bv_j \) serve to provide concrete mathematical objects in the proofs below. The specific basis chosen does not affect the validity of the results.
Next, for \( \varepsilon > 0 \), define the \textit{thickened polytope}:
\begin{equation}\label{def_ext_to_ndim}
    P^\varepsilon \,:=\, P + \varepsilon C(P^\perp).
\end{equation}

Let \( j := \dim(P)-1 \). Then  the $j-$faces of \( P \), (the facets of $P$) become subsets of \(\partial P^\varepsilon \). We define
\[
\bF_0^{d-1} (P^\varepsilon) := \{ F + \varepsilon C(P^\perp) : F \in \bF^{j}(P) \}
\]
to be the set of \((d-1)\)-dimensional facets of \( P^\varepsilon \) arising from extrusion of the \( j \)-faces of \( P \). The remaining \((d-1)\)-faces of $P^\varepsilon$ are collected in
\[
\bF_1^{d-1} (P^\varepsilon) := \bF^{d-1}(P^\varepsilon) \setminus \bF_0^{d-1} (P^\varepsilon).
\]

\begin{figure}[h!]
\centering
\begin{tikzpicture}[scale=2.0, line join=round]

\coordinate (centroid) at (0,0);

\coordinate (A) at (-1.2,-0.6);
\coordinate (B) at (1.2,-0.6);
\coordinate (C) at (0,1.2);
\node[left] at (A) {\footnotesize$A$};
\node[right] at (B) {\footnotesize$B$};
\node[above] at (C) {\footnotesize$C$};
\def\eta{0.7}

\coordinate (As) at ($(centroid)!\eta!(A)$);
\coordinate (Bs) at ($(centroid)!\eta!(B)$);
\coordinate (Cs) at ($(centroid)!\eta!(C)$);

\def\e{0.25}
\coordinate (shift) at (0.26,0.1);  

\coordinate (As_up) at ($(As)+(shift)$);
\coordinate (Bs_up) at ($(Bs)+(shift)$);
\coordinate (Cs_up) at ($(Cs)+(shift)$);

\coordinate (As_down) at ($(As)-(shift)$);
\coordinate (Bs_down) at ($(Bs)-(shift)$);
\coordinate (Cs_down) at ($(Cs)-(shift)$);

\fill[gray!20] (As_down) -- (Bs_down) -- (Cs_down) -- cycle;
\fill[gray!20] (As_up) -- (Bs_up) -- (Cs_up) -- cycle;
\fill[gray!20] (As_down) -- (As_up) -- (Cs_up) -- (Cs_down) -- cycle;
\fill[gray!20] (Bs_down) -- (Bs_up) -- (Cs_up) -- (Cs_down) -- cycle;

\draw[thick] (A) -- (B) -- (C) -- cycle;

\draw[thick, dashed] (As) -- (Bs) -- (Cs) -- cycle;

\draw[thick] (As_up) -- (Bs_up) -- (Cs_up) -- cycle;
\draw[thick] (As_down) -- (Bs_down) -- (Cs_down) -- cycle;

\draw[dashed] (As_up) -- (As_down);
\draw[dashed] (Bs_up) -- (Bs_down);
\draw[dashed] (Cs_up) -- (Cs_down);
\filldraw[black] (centroid) circle (0.5pt);
\node[below right] at (centroid) {\footnotesize$\mathbf{b}_P$};

\end{tikzpicture}
\caption{
The triangle $A,\, B,\, C$, defines the polytope \(P\), (solid) is centered at its arithmetic mean \(\mathbf{b}_P\). The scaled polytope \(P_\eta\) (dashed) contracts toward the arithmetic mean with factor \(\eta\). The prism \((P_\eta)^\varepsilon = P_\eta + \varepsilon C(P^\perp)\) is shown with light gray faces.
}
\end{figure}


\begin{theorem}\label{theo_hopf_miranda_formal}
Let $f$ be the continuous Hopfield function defined in (\ref{networkFunHopf}), with pairwise distinct, unit-length vectors $\bom_1,\ldots,\bom_n\in \R^d$ as the columns of $\bW$. For given $J\subset[n]$ assume $\{\bom_j:\, j\in J\}$ are 
\CDPS\, and that for $P:=\conv\{\bom_j:\, j\in J\}$ and all facets $F\in \bF^{\dim(P)-1}(P)$, the vectors $\{\bom_j:\, j\in I(F)\}$ are \CDPS. Then there are $1>\eta>0$ and $\beta(\varepsilon)>0$ with $\beta(\varepsilon)\rightarrow \infty$ as $\varepsilon\rightarrow 0$, such that 
\begin{equation*}
    (P_\eta)^{\varepsilon} \cap \, \fix(f)\,\neq \, \emptyset.
\end{equation*}
Further, provided $\vert J\vert >1$, there exists a $\beta_0>0$ such that for $\beta(\varepsilon)>\beta_0$, all fixed points in $(P_\eta)^{\varepsilon}$ are unstable.
\end{theorem}

\cref{theo_hopf_miranda} follows directly from \cref{theo_hopf_miranda_formal}. Indeed, \cref{theo_hopf_miranda}~(1) is an informal version of \cref{theo_hopf_miranda_formal}. For \cref{theo_hopf_miranda}~(2), note that every \( k \)-face \( F \) of \( P \) itself satisfies the conditions of \cref{theo_hopf_miranda_formal} and thus admits a fixed point in \( (F_{\eta_F})^{\varepsilon_F} \) for suitable \( \eta_F > 0 \), \( \varepsilon_F > 0 \). To ensure that each face contributes a distinct fixed point, one observes that \( \varepsilon_F \) can be chosen sufficiently small such that the $(F_{\eta_F})^{\varepsilon_F}$ are mutually disjoint for different faces.

In this way, every subset of pattern vectors whose convex hull satisfies the conditions of \cref{theo_hopf_miranda_formal} gives rise to a fixed point for sufficiently large \( \beta \). In particular, the stable fixed points near each vertex \( \bom_j \), \( j=1,\ldots,n \) (i.e., the 0-faces of \( \conv\{\bom_1,\ldots,\bom_n\} \)), do not coincide with the fixed points associated with higher-dimensional faces guaranteed by \cref{theo_hopf_miranda_formal}. The existence of these attractive fixed points in the vicinity of each \( \bom_j \) is shown in \cite[Lemma A6]{ramsauer2020hopfield} for sufficiently large \( \beta \); actually, the authors in \cite{ramsauer2020hopfield} fix \( \beta = 1 \) but assume that the norm of the vectors \( \bom_j \) are sufficiently large.

\begin{corollary}\label{cor_allCDPS}
Let \( f \) be the continuous Hopfield function defined in (\ref{networkFunHopf}), with pairwise distinct, unit-length vectors \( \bom_1,\ldots,\bom_n \in \R^d \) as the columns of \( \bW \). Assume that for every face $F\in \bF^k(P)$, $0\leq k \leq\dim(P)$, where \( P := \conv\{\bom_1,\ldots,\bom_n\} \), the corresponding set of vectors $\{\bom_j:\ j\in I(F)\}$ is \CDPS. Then for sufficiently large \( \beta > 0 \), there are at least as many fixed points as facets.
\end{corollary}

\begin{remark}\label{rem_allCDPS}
Under the condition of \cref{cor_allCDPS}, in higher dimensions there are typically many more unstable fixed points than patterns to be memorized, i.e., the \( \bom_1,\ldots,\bom_n \). The precise number depends on the geometric configuration of the $\bom_1,\ldots,\bom_n$ and the relationship between the number of vertices \( n \) and the total number of faces of \( \conv\{\bom_1,\ldots,\bom_n\} \).
\end{remark}

We next provide proofs and auxiliary results.


\begin{proof}(\cref{lem_hopf_conv})
For the sake of brevity, let $\taup:=\taup(\bv),\, \taum:=\taum(\bv)$ and $k:=\vert J\vert
$. For $j\notin J$ we obtain
\begin{equation}\label{inq_lower_res_w2}
    \exp(\beta\, \bom_j^T \bx)\, \leq  \, \exp(\beta\,( \taum+\varepsilon)).
\end{equation}
W.l.o.g. let $1\in J$ and $\taup=\bom_1^T \bv$, then
\begin{eqnarray}\label{inq_upper_res_w1}
   \exp(\beta\, \bom_1^T \bx) 
   \,\geq\, \exp\left(\beta(\taup-\varepsilon)\right).  
\end{eqnarray}
For $Z:=\sum_{j=1}^n \exp(\beta\, \bom_j^T \bx)$, it follows from (\ref{inq_upper_res_w1}) that
\begin{equation}\label{inq_Z}
  \exp\left(\beta(\taup-\varepsilon)\right)\,\leq\, Z.  
\end{equation}
Now, let $f(\bx)=\bv_1+\bv_2$, where
\begin{equation*}
 \bv_1\,:=\,\sum_{j\in J} \bom_j\, \frac{\exp(\beta\, \bom_j^T \bx )}{Z}, \quad \bv_2\,:=\,\sum_{j\notin J} \bom_j\, \frac{\exp(\beta\, \bom_j^T \bx )}{Z}.
\end{equation*}
Then (\ref{inq_lower_res_w2}) and (\ref{inq_Z}) imply 
\begin{equation}\label{inq_w2_upper}
    \Vert \bv_2 \Vert\, \leq\, \frac{ (n-k) \exp(\beta (\taum+\varepsilon))}{ \exp\left(\beta(\taup-\varepsilon)\right)}= (n-k)\exp(\beta (-\delta+2\varepsilon)).
\end{equation}
From (\ref{inq_lower_res_w2}) and (\ref{inq_Z}), we further obtain that 
\begin{equation}\label{upper_koverZ}
\begin{split}
        q\,:=\,\sum_{j\in J} \frac{\exp(\beta \bom_j^T\bx)}{Z} &=1-\frac{\sum_{j\notin J}\exp(\beta\, \bom_j^T\bx)}{Z}\\
    &\geq 1 - (n-k)\,\exp(\beta( -\delta+2\varepsilon)).
\end{split}
\end{equation}
Since $\bv^*:=1/q\, \bv_1\in \conv\{\bom_j:j\in J\}$, it follows that $\Vert \bv^*-\bv_1\Vert\geq \dist(\bv_1,\conv\{\bom_j:j\in J\})$ and the triangle inequality further gives 
\begin{equation*}
    \Vert \bv^*-\bv_1\Vert\,\leq\, \left(\frac{1}{q}-1\right) \frac{\sum_{j\in J} \exp(\beta \bom_j^T\bx)}{Z} \,=\, (1-q),
\end{equation*}
and thus $1-q\geq  \dist(\bv_1,\conv\{\bom_j:j\in J\})$. Hence, (\ref{upper_koverZ}) implies
\begin{equation} \label{inq_w1_upper}
    \dist(\bv_1,\conv\{\bom_j:j\in J\}) \,\leq \,(n-k)\,\exp(\beta( -\delta+2\varepsilon)).
\end{equation}
Finally, considering that $\dist\left(f(\bx),\conv\{\bom_j:j\in J\}\right)\leq \Vert \bv^*-\bv_1\Vert+\Vert \bv_2\Vert$, (\ref{inq_w2_upper}) and (\ref{inq_w1_upper}) give
\begin{equation*}
    \dist\left(f(\bx),\conv\{\bom_j:j\in J\}\right)\,\leq\, 2\,(n-k)\exp(\beta (-\delta+2\varepsilon)).
\end{equation*}
\end{proof}

\begin{lemma}
\label{lem_uniform_instab}
Let $f$ be the continuous Hopfield function defined in (\ref{networkFunHopf}), with pairwise distinct, unit-length vectors $\bom_1,\ldots,\bom_n\in \R^d$ being the columns of $\bW$. Let $\bx^{*}\in \fix(f)$ such that  $\Vert \bx^{*} - \bom_j\Vert>\delta>0$ for all $j=1,\ldots,n$. Then there exists an $a:=a(\delta)>0$ such that for
\[
    \beta_0
    \;:=\;
    \frac{4}{a(1-\omega)^2},\quad \mathrm{where}\  \max_{j\neq k}\vert \bom_j^T \bom_k\vert =:\omega \, <\, 1,
\]
the following holds. If $\beta >\beta_0$, then $\bx^{*}$ is an unstable fixed point. 

\end{lemma}

\begin{proof}(\cref{lem_uniform_instab})

By assumption, every $\bom_j$ defines a unique vertex of $P:=\conv\{\bom_1,\ldots, \bom_n\}$ and thus $h:\Delta^{n-1}\rightarrow P$, $h(\bp)=\bW\bp$ maps the vertices of $\Delta^{n-1}$ one to one to the $\bom_1,\ldots, \bom_n$. By continuity of $h$ and compactness of $\Delta^{n-1}$, we thus find an $a>0$, such that for all $\bp\in \Delta^{n-1}$ that satisfy $\Vert \bW \bp - \bom_j\Vert >\delta$, we have $\bp^{(j)}\geq a$ for at least two $j\in [n]$. 
 Since $\bx^{*}\in\fix(f)$, there is a unique $\bp^{*}\in\Delta^{n-1}$  with $\bx^{*}=\bW \bp^{*}$ and $\bp^{*}=\sm_\beta(\bW^T\bx^{*})$. The Jacobian of $f$ at $\bx^{*}$ is then given by 
\begin{equation}\label{JacHopf}
    \bJ(\bx^*)\,=\,\beta\bW \bS(\bp^{*})\bW^T,\quad   \text{where}\    \bS(\bp)
    \;:=\;
    \bD(\bp)-\bp\bp^T.
\end{equation}
We use the following second moment decomposition
\[
  \bW \bS(\bp) \bW^T 
  = \sum_{i=1}^n \bp^{(i)}\, (\bom_i - \bmu)(\bom_i - \bmu)^T),
  \quad
  \text{where} \quad
  \bmu := \bW \bp,
\]
cf. \cite[Equation (74-78)]{ramsauer2020hopfield}.
Setting $p^*_k:=(\bp^{*})^{(k)}$, $k=1,\ldots,n$, we have
\[
  \bJ(\bx^*) = \beta\, \sum_{i=1}^n p_i^*\, (\bom_i - \bmu)(\bom_i - \bmu)^T.
\]
For a pair $j\neq k$ with $p^*_j, p^*_k \ge a$, define
\[
  \bu := \frac{\bom_j - \bom_k}{\| \bom_j - \bom_k \|} \in \R^d.
\]
Then,
\begin{equation}\label{qFunstab}
   \bu^T\, \bJ(\bx^*)\, \bu
  = \beta\, \sum_{i=1}^n p^*_i\, [(\bom_i - \bmu)^T \bu]^2
  \;\ge\; \beta\, \Big( 
    p^*_j \left((\bom_j - \bmu)^T \bu\right)^2
    + p^*_k \left((\bom_k - \bmu)^T \bu\right)^2 
  \Big).   
\end{equation}

Considering $\Vert \bom_i\Vert=1$, $i=1,\ldots,n$, we have 
\begin{eqnarray}
    (\bom_j-\bmu)^T (\bom_j-\bom_k)&=1- \bom_j^T\bom_k- \bmu^T(\bom_j-\bom_k)\label{eq1unstab}, \\
      (\bom_k-\bmu)^T (\bom_j-\bom_k)&=-1+ \bom_j^T\bom_k- \bmu^T(\bom_j-\bom_k)\label{eq2unstab}.
\end{eqnarray}
Going through all possible sign configurations for $\bmu^T(\bom_j-\bom_k)$ and $\bom_j^T\bom_k$, it is verified that the absolute value of least one of (\ref{eq1unstab}) and (\ref{eq2unstab}) is greater or equal to $(1-\omega)$. Together with $\Vert \bom_j-\bom_k\Vert \leq 2$ and recalling that $p^*_j, p^*_k \ge a$, the left hand side in (\ref{qFunstab}) can be further estimated to
\begin{equation*}
    \bu^T\, \bJ(\bx^*)\, \bu\, \geq \, \frac{\beta}{4} a (1-\omega)^2.
\end{equation*}
This concludes the proof.

\end{proof}

\begin{proof}(\cref{theo_hopf_miranda_formal})

Let $l:=\dim(P)$. By assumption the vectors corresponding to $P$ and to every $F\in \bF^{l-1} (P)$ are \CDPS\, and this implies that for all index sets $I\in (\{J\} \cup \{I(F):F\in \bF^{l-1} (P)\})=\mathcal{I}$, we have
\begin{equation*}
    \delta_I\,:=\,\min_{\bv\in \conv\{\bom_j:j\in I\}}\left(
\max_{j\in I}\bom_j^T \bv  -\max_{j\in [n]\setminus I}\bom_j^T \bv\right)\, >0.
\end{equation*} 
We define
\begin{equation*}
0\,<\,\delta\,:=\, \min_{I\in \mathcal{I}}\delta_I.
\end{equation*}
For the determination of $\eta$ and $\varepsilon$, first consider some arbitrary $F\in \bF^{l-1}(P)$ and the corresponding
\begin{equation}\label{def_tF}
    \tilde{F}:=\eta\, (F-\bb_P )\, +\bb_P\, +\varepsilon\, C(P^\perp) \, \in \bF_0^{d-1}((P_\eta)^{\varepsilon}).
\end{equation}
Then every $\bx\in \tilde{F}$ decomposes into
\begin{equation*}
    \bx\,=\,\eta\, (\bv- \bb_P)\,+\,\bb_P \,+\, \varepsilon\, \bu
\end{equation*}
such that  $\eta(\bv-\bb_P)$, with $\bv\in F$, is the orthogonal projection of $\bx-\bb_P$ on $P-\bb_P$, and $\varepsilon \bu\,=\bx-\bb_P - \,\eta(\bv-\bb_P)$ with $\bu\in  C(P^\perp )$.
Then the condition $\Vert \bom_j\Vert=1$ for $j=1,\ldots,n$ and the fact that they are pairwise different imply
\begin{equation*}
   \Vert \bv - \bb_P\Vert\, \leq \, 2\quad \mathrm{and}\quad 0<\,q\,:=\,\min\limits_{\bv\in \partial P}\Vert \bv - \bb_P\Vert.
\end{equation*}
The definition of $C(P^\perp)$ in (\ref{def:BPperp}) implies $\Vert \bu \Vert\, \leq\, 1$.
Thus for $\br':=\bx-\bv$, $\mu:=\Vert \br'\Vert$, $\br:=1/\mu \,\br'$, the latter and orthogonality of $\bu$ and $\bv-\bb_p$ give
\begin{equation}\label{ineq_mu}
   q(1-\eta) \,\leq \, \mu \leq\, \sqrt{4(1-\eta)^2\, +\, \varepsilon^2}.
\end{equation}
Note that by construction,
\begin{equation}\label{muDist}
    \mu\,=\,\dist(\bx, F).
\end{equation}
We can now choose $1>\eta>0$ sufficiently close to $1$ and $\varepsilon>0$ sufficiently small to ensure that 
\begin{equation*}
    \sqrt{4(1-\eta)^2\, +\, \varepsilon^2}\,\leq\, \frac{\delta}{3}.
\end{equation*}
Then for $F\in \bF^{l-1}(P)$ and corresponding $\tilde{F}\in \bF_0^{d-1}((P_\eta)^{\varepsilon})$, cf. (\ref{def_tF}), the estimate in (\ref{ineq_mu}) implies that we can write $\bx\in \tilde{F}$ as
\begin{equation}\label{def_eta_eps}
    \bx\,=\,\bv\, +\, \mu\,\br \quad \mathrm{with}\ \bv\in F,\ \Vert\br\Vert=1,\   q(1-\eta)\,\leq \mu\,\leq \,\frac{\delta}{3}.
\end{equation}
This allows to chose $\tilde{\beta}>0$ so large that for $\beta\geq \tilde{\beta}$ and all $F\in \bF^{l-1}(P)$ with corresponding $I:=I(F)$, we have
\begin{equation*}
    2(n-\vert I\vert)\,\exp\left(\beta (-\delta + 2\mu)\right)\,<\,\mu \quad \mathrm{for \, all}\quad q(1-\eta)\,\leq \mu\,\leq \,\frac{\delta}{3}.
\end{equation*}
Now, together with \eqref{def_eta_eps} and \eqref{muDist}, the latter ensures that we can apply \cref{lem_hopf_conv} to conclude, that for all $F\in \bF^{l-1}(P)$ and the corresponding face $\tilde{F}\in \bF_0^{d-1}(P_\eta)^{\varepsilon}$ as in (\ref{def_tF}), we have
\begin{equation}\label{F_attract}
    \dist(f(\bx),F)\,\leq\,  \dist(\bx,F)\quad \mathrm{for\, all}\quad \bx\in \tilde{F}.
\end{equation}
Next let  $F\in\bF_1^{d-1}(P_\eta)^{\varepsilon}$ and $m:=d-\dim(P)$, then by the definition of $C(P^\perp)$ in (\ref{def:BPperp}), we have 
\begin{equation*}
  \frac{\varepsilon}{\sqrt{d}}\, \leq \,\frac{\varepsilon}{\sqrt{m}}  \, \leq \, \dist(\bx,P) \, \leq  \,\varepsilon \quad \mathrm{for\, all}\quad \bx\in F.
\end{equation*}
The lower estimates follow from the fact that the vectors in $\partial C(P^\perp)$ are convex combinations of $m$ ONB vectors.
We can then assume $\delta_J>2\varepsilon$, as we can further decrease $\varepsilon$ if needed. In fact, the lower bound on $\mu$ in \eqref{ineq_mu} does not depend on $\varepsilon$. For such $\varepsilon$, we can then choose $\beta(\varepsilon)>\tilde{\beta}$ sufficiently large, to ensure that 
\begin{equation*}
    2(n-\vert J\vert)\,\exp\left(\beta (-\delta_J+2 \varepsilon')\right)\,<\,\varepsilon'\quad \mathrm{for \, all}\quad \frac{\varepsilon}{\sqrt{d}}\leq \varepsilon'\leq \varepsilon.
\end{equation*}
\cref{lem_hopf_conv} then ensures that for all $F\in \bF_1^{d-1}((P_\eta)^{\varepsilon})$ we have
\begin{equation}\label{P_attract}
    \dist(f(x),P)\,\leq \, \dist(x,P)\quad \mathrm{for\, all}\quad \bx\in F.
\end{equation}
Note that the above $\beta(\varepsilon)$ can be chosen to be increasing as $\varepsilon\rightarrow 0$. 

Now the conditions of \cref{prop_miranda} apply to $(P_\eta)^{\varepsilon}$ and $f:(P_\eta)^{\varepsilon}\rightarrow \R^d$. Indeed, with $C_2:=\bF_0^{d-1}((P_\eta)^{\varepsilon})$ and $C_1:=\bF_1^{d-1}((P_\eta)^{\varepsilon})$, (\ref{F_attract}) yields that condition (\ref{mirandaCondition_02}) holds, and (\ref{P_attract}) yields that (\ref{mirandaCondition_01}) holds.
Thus, \cref{prop_miranda} provides a fixed point of $f$ in $(P_\eta)^{\varepsilon}$. 

Finally, if $\vert J\vert > 1$, we have that the fixed point in $(P_\eta)^\varepsilon$ from above has lower bounded distance to the vertices $\bom_j$ for $j\in J$, determined by $\eta, \, \varepsilon$. As it turns out from the proof, for further increasing $\beta$ we can keep these $\eta, \, \varepsilon$ fixed and still have our fixed point in $(P_\eta)^\varepsilon$.  By the \CDPS\, condition it is clear that we also have lower bounded distance to the $\bom_j$ for $j\in[n]\setminus J$, independent of $\beta$. Thus, the condition of \cref{lem_uniform_instab} holds true so that the fixed points $(P_\eta)^\varepsilon$ are unstable for sufficiently large $\beta>0$ and we hence find a $\beta_0\geq \tilde{\beta}$ so that the latter holds for all $\beta>\beta_0$
\end{proof}

The proof of \cref{prop_miranda} employs key concepts from topological degree theory \cite{milnor1997topology, nagumo1951theory,jezierski2006homotopy}. We briefly illustrate the main concepts needed.
Consider a compact domain $K \subset \R^d$, and let $f: K \rightarrow \R^d$ be a continuously differentiable mapping. For some $\by \notin f(\partial K)$, $\by$ being a regular point, the \textbf{Brouwer degree} is defined as
\begin{equation*}\label{deg_brouwer}
    \deg(f, K, \by) \,=\, \sum_{\bx \in f^{-1}(\by)} \mathrm{sign}\left(\det(\Jac_{f}(\bx))\right),
\end{equation*}
where  $f^{-1}$ denotes the pre-image and $\Jac_f$ is the Jacobian of $f$.
Let $H: K \times [0,1] \rightarrow \R^d$ be continuous, such that $f_t := H(\cdot, t): K \rightarrow \R^d$ is continuously differentiable, and $\by \notin f_t(\partial K)$ for all $t \in [0,1]$. The \textbf{homotopy invariance} then asserts that $\deg(f_0, K, \by) = \deg(f_1, K, \by)$. The notion of degree can be extended to continuous functions, cf.\cite{nagumo1951theory, jezierski2006homotopy}, ensuring that properties such as homotopy invariance are preserved.

\begin{proof}(\cref{prop_miranda})
Without loss of generality we can assume that the subspace $U_1$ and $U_2$ are axis aligned. There hence are disjoint index sets $I_1\cup I_2= [d]$, such the standard unit vectors corresponding to $I_k$ are an orthonormal basis of $U_k$ for $k=1,2$. We can also assume that the origin $0$ is an interior point of $P$. The above assumptions can be realized by an orthogonal change of coordinates and translation that do not affect the assertion. 
For every facet $F\in\bF^{d-1}(P)$, we then have $b_F \,>\, 0$ as it follows from the convexity of $P$ together with the fact that the $\bn_F$ are outward oriented (w.r.t. $P$) and that $0$ is an interior point of $P$.

We observe that for every $F\in C_1$, condition \eqref{mirandaCondition_01} implies
\begin{equation}\label{cond1_PM_proof}
    \bn_F^T\, f(\bx)\, \leq \,\bn_F^T\,\bx \,> 0,\ \text{for all}\, \bx\in F.
\end{equation}
In the same way, for every $F\in C_2$ condition \eqref{mirandaCondition_02} yields
\begin{equation}\label{cond2_PM_proof}
    \bn_F^T\, f(\bx)\,  \geq\,\ \bn_F^T\,\bx \,>\, 0,\ \text{for all}\, \bx\in F.
\end{equation}
We next construct a homotopy where we keep control over the boundary $\partial P$ by means of the before \eqref{cond1_PM_proof},\eqref{cond2_PM_proof}. To this end, let $g(\bx) := f(\bx) - \bx$. If $g$ has a zero on the boundary $\partial P$, the assertion follows and we are done. So we assume $g(\bx)\neq 0$ for all $\bx\in \partial P$ for the following. Define $h:\R^d\to\R^d$ coordinate wise by
\begin{equation}
h(x)^{(j)} :=
\begin{cases}
-x^{(j)} & \text{for}\  j\in I_1,\\
\ \ x^{(j)} & \text{for}\  j\in I_2.
\end{cases}
\end{equation}
For $F\in C_1$ we have $\bn_F^{(j)}=0$ for all $j\in I_2$. This follows from $\bn_F\in U_1$ together with orthogonality and the assumptions that $U_1,\,U_2$ are axis aligned. With the second inequaltity in \eqref{cond1_PM_proof} and by the definition of $h$, we thus have $\bn_F^T \,h(\bx) < 0$ for all $\bx\in F$ and all $F\in C_1$. 
Hence for all $F\in C_1$, \eqref{cond1_PM_proof} gives 
\begin{equation*}
    \bn_F^T\,\left((1-t)\,g(\bx)\,+\,t\, h(\bx)\right)\,\leq\, 0
\end{equation*}
for all $t\in [0,1]$ and all $\bx\in F$, with strict inequality for $t>0$. In the similar way, for all $F\in C_2$, \eqref{cond2_PM_proof} yields 
\begin{equation*}
    \bn_F^T\,\left((1-t)\,g(\bx)\,+\,t\, h(\bx)\right)\,\geq\, 0
\end{equation*}for all $t\in [0,1]$ and all $\bx\in F$, with strict inequality for $t>0$.
We have thus shown that the homotopy $H : P \times [0,1] \rightarrow \mathbb{R}^d$ defined by
\begin{equation*}
    H(\mathbf{\bx}, t) = (1 - t)\,g(\mathbf{\bx}) +\, t \, h(\mathbf{\bx}),
\end{equation*}
does not vanish for all $\bx\in \partial P$ and all $t \in (0,1]$. With the above assumption that $g(\bx)\neq 0$ for $\bx\in \partial P$, we also have $H(\mathbf{\bx}, 0)\neq 0$ for $\bx\in \partial P$. Hence, since $h(0) = 0$ and $\deg(h,K,0)\neq 0$, the homotopy invariance \cite{milnor1997topology,jezierski2006homotopy} yields $\deg(g,K,0)\neq 0$, which guarantees that $g(\bx) = 0$ has a solution $\bx^* \in K$. This is equivalent to $f(\bx^*) = \bx^*$ an thus the assertion follows.
\end{proof}

\section{Numerical Considerations}\label{sec:num_experiments}

In this section, we first numerically illustrate that our central \CDPS,  condition is not vacuous and occurs under simple geometric random models. We generate random polytopes by applying linear transformations with condition numbers
$\kappa\in\{2,4,6\}$ to the standard basis vectors in $\R^n$ (we have $d=n)$ here), for $n\in\{20,50\}$.
For each configuration, we sample $100$ faces with $k\in\{4,7,15\}$ vertices.
For a sampled face and its facets (to meet the assumption in \cref{theo_hopf_miranda}), we estimate the separation margin $\delta$ from \cref{lem_hopf_conv} by Monte Carlo sampling:
we draw $\bv\in\conv\{\bom_j:j\in J\}$ via random convex combinations and record
\[
\delta(v)\;:=\;\max_{j\in J}\bom_j^T v\;-\;\max_{j\notin J}\bom_j^T v.
\]
We take the minimum over $10{,}000$ samples for each face and its facets, respectively,
and then record the minimum over all tested faces as an estimator of the configuration margin.
We count the \CDPS\, condition to be satisfied if this estimate exceeds $10^{-10}$.

\begin{table}[h]
\centering
\small
\setlength{\tabcolsep}{5pt}
\begin{tabular}{ccccccc}
\hline
$n$ & $\kappa$ & $\delta_{\min}$ & $\delta_{\text{median}}$  & \CDPS\, rate  \\
\hline
20 & 2 & $3.34\times 10^{-2}$ & $9.51\times 10^{-2}$  & $100\%$ \\
20 & 4 & $-4.03\times 10^{-2}$ & $2.26\times 10^{-2}$  & $80\%$ \\
20 & 6 & $-4.67\times 10^{-2}$ & $3.72\times 10^{-3}$  & $55\%$ \\
\hline
50 & 2 & $4.35\times 10^{-2}$ & $1.09\times 10^{-1}$ &  $100\%$ \\
50 & 4 & $-6.23\times 10^{-3}$ & $6.71\times 10^{-2}$ &  $99\%$ \\
50 & 6 & $-1.86\times 10^{-3}$ & $5.54\times 10^{-2}$ &  $96\%$ \\
\hline
\end{tabular}
\caption{Numerical evaluation of the \CDPS\, condition on distorted polytopes.
For each configuration we sample $100$ faces for each $k\in\{4,7,15\}$.
Reported values are the minimum and median of the estimated margins
over all tested faces and their facets.}
\label{tab:numerical_summary}
\end{table}
For well-conditioned transformations ($\kappa=2$), the \CDPS\, condition holds for all sampled faces.
The rate decreases with increasing distortion but remains substantial in these dimensions.

To estimate the magnitude of sufficient $\beta$ for concluding the existence of 
fixed points as in \cref{theo_hopf_miranda_formal}, we consider isolated polytopes $P=\conv\{\bom_1,\ldots,\bom_k\}$ by taking $\bW$ to consist of these vectors as columns, so the dynamics of $f$ remains within $P$. 
We perform a grid search over $(\beta,\eta)$ where, for each facet of $P_\eta$ (cf.~\eqref{def_pmu}), we sample $10{,}000$ points and verify whether all are mapped toward the supporting affine space of the corresponding facet 
of the original $P$.
If this holds for all facets, $P_\eta$ satisfies (up to a homeomorphism) the outward mapping condition (\ref{mirandaCondition_02}) required for the Poincar\'e--Miranda argument.
We record the smallest $\beta$ exhibiting this behavior.
Across $k\in\{4,7,15\}$ and $n\in\{20,50,500\}$, such behavior is typically observed for $\beta\in[5,20]$.

The proof of \cref{theo_hopf_miranda_formal} follows the same line of reasoning as the one underlying the above $\beta$-search. More precisely, the Poincar\'e--Miranda conditions (\ref{mirandaCondition_01},\ref{mirandaCondition_02}) in \cref{prop_miranda} for $f$ (up to a homeomorphism) for the points on the facets of $P_\eta^\varepsilon$ (cf.~\eqref{def_eta_eps}) are deduced from \eqref{dist_estimate} from \cref{lem_hopf_conv}. However, enforcing this worst-case sufficient condition, can result in large values of $\beta$, far beyond what is actually needed in concrete cases. Indeed, in contrast to the numerical estimates for $\beta$, on the same geometrical configurations, the sufficient $\beta$ derived by means of \cref{lem_hopf_conv} to obtain the outward mapping condition on some $P_\eta$ as above, range from approximately $12$ to about $400$. This indicates room for improvement towards quantitative rigorous bounds.

For the standard simplex case ($\bW=\bI$), the complete fixed point structure of the softmax functions is known (see \cite{tivno2009bifurcation}). In particular, nontrivial fixed points emerge once
$\beta$ exceeds certain thresholds (see \ref{sec:append}), and for $\beta>n$ the full set of $2^n-1$ fixed points is present
\cite{tivno2009bifurcation}. The numerical observations above suggest that for moderately distorted configurations the values of $\beta$ required for the existence of additional fixed points remain of comparable order.

The standard scaling in the attention mechanism of transformer networks
(\ref{scaledAttention}) uses the factor $1/\sqrt{d_k}$, where $d_k$ is the internal dimension of the key vectors (see Equation~1 in \cite{vaswani2017attention}).
In \cite{widrich2020modern}, this scaling is also used for modern continuous Hopfield networks, or the scaling factor is left as a trainable parameter \cite{ramsauer2020hopfield}. These factors are not directly comparable to our parameter $\beta$, since we assume unit-length patterns $\|\bom_j\|=1$. However, under the heuristic assumption that vectors entering the softmax have stochastically independent entries following approximately a standard Gaussian distribution (cf. page 4 in \cite{vaswani2017attention}), their Euclidean norm grows as $\sqrt{n}$, where $n$ is the sequence length. The numerically determined values for $\beta$ are thus comparable in order of magnitude to the effective scaling in \cite{widrich2020modern} ($n=10{,}000$, cf. Table A3) and \cite{vaswani2017attention} ($n\in\{1024,4096\}$, cf. Table~3). 
Within the scope of this theoretical work, we do not attempt a detailed comparison with practical applications, but we note that the numerically determined $\beta$ values (typically 5-20) are not extreme.

\section{Conclusion}\label{sec:concl}

In this work, we prove the existence of unstable fixed points for modern continuous Hopfield networks 
\[
f(\bx) = \bW\,\sm_\beta(\bW^T \bx),
\quad \text{see~(\ref{networkFunHopf})},
\]
under the \CDPS~condition (\cref{def_CDPS}) and for sufficiently large $\beta > 0$. This finding connects the exact analysis of the fixed points of $\sm_\beta$ from \cite{tivno2009bifurcation} with the research on modern continuous Hopfield networks \cite{ramsauer2020hopfield}. 
Our work leaves several questions open that naturally arise when comparing the basic case of $\sm_\beta$, treated in \cite{tivno2009bifurcation}, with $f$ as in (\ref{networkFunHopf}). For instance, the proof of \cref{theo_hopf_miranda_formal} does not imply that the fixed points located in the sets $(P_\eta)^{\varepsilon}$ are unique, since \cref{prop_miranda} provides only an existence result. We also give no further information about the exact positions of these unstable fixed points.  To give rigorous answers to such questions would require further considerations and may necessitate additional conditions beyond \CDPS. We postpone these questions to future work.
Results in these directions would contribute to a more complete understanding of the dynamics of modern Hopfield networks and clarify how attractive and unstable fixed points determine the shape of orbits and basins of attraction. The examples in Figure~\ref{fig:netdynamics} suggest that orbits may pass near unstable fixed points located close to higher-dimensional faces before converging to the intended stable fixed points near the stored patterns $\bom_1,\ldots,\bom_n$. Our results in \cref{lem_hopf_conv} and 
\cref{theo_hopf_miranda_formal} provide a mechanism explaining such behavior when the \CDPS\, condition is satisfied. The discussion in Section~\ref{sec:num_experiments} indicates that this geometric condition 
holds for a substantial fraction of randomly sampled faces in moderately 
distorted configurations.

\bibliographystyle{plainnat} 
\bibliography{references}

\appendix
\section{The Fixed Points of the Softmax Mapping $\sm_\beta$}\label{sec:append}

We briefly review some results from \cite{tivno2009bifurcation} on the emergence of fixed points of $\sm_\beta$ as $\beta$ increases. 
It should be noted that the results in \cite{tivno2009bifurcation} are formulated in terms of $1/\beta=:T$, which is interpreted as a temperature. We have chosen the results in terms of $\beta$ as in \cite{ramsauer2020hopfield}.

Let $\log$ denote the natural logarithm in the sequel.

It is directly seen that the center $1/n\,\bi$ of $\Delta^{n-1}$ is always a fixed point of $\sm_\beta$, i.e. $\sm_\beta(1/n\,\bi)=1/n\,\bi$. 

We next introduce some notation to formalize line segments in $\Delta^{n-1}$ through its center $1/n\bi$, that connect two faces of dimension $k,l$ with $k+l=n-1$ being on opposite sides w.r.t $1/n \bi$, cf. \cite[Figure 1]{tivno2009bifurcation}.

For $J\subset [n]$, $J\neq [n],\emptyset$ and $k:=\vert J\vert$, we set
\begin{equation}\label{def_K}
\begin{split}
    l_J&:[0,1/k]\rightarrow \Delta^{n-1}\\[0.2cm]
    l_J(x)&:=\bx \ \text{with }\begin{cases}
        \bx^{(j)}=x\quad &\text{if}\, j\in J\\[0.2cm]
        \bx^{(j)}=\frac{1-kx}{n-k}\quad &\text{if}\, j\in [n]\setminus J,\
    \end{cases}\\[0.2cm]
    L_J&\,:=\,l_J([0,1/k]),\\[0.2cm]
    \bx^{(J)}&\,:=\,l_J^{-1}(\bx),\ \mathrm{for}\ \bx\in L_J, \quad (\text{i.e.}\ \bx^{(J)}=\bx^{(j)} \,\mathrm{for}\, j\in J).
\end{split}
\end{equation}

We also set $L_{\emptyset}=L_{[n]}:=\left\{1/n\bi \right\}$.
It follows immediately that 
\begin{equation}\label{L_JtoL_J}
    \sm_\beta(L_J)\,\subset\, L_J \quad\mathrm{for\, all}\ \, \beta>0.
\end{equation}
The following observation, cf. \cite[Theorem 2.1]{tivno2009bifurcation}, constitutes the starting point of the fixed point analysis of $\sm_\beta$:
\begin{equation*}
        \fix(\sm_\beta)\,\subset\, \bigcup_{J\subset [n]} L_{J}.
\end{equation*}   

Let us briefly sketch how the analysis of fixed points evolve from there. Let $k:=\vert J\vert$. 
With the above notation and for some $\bx\in L_J$ and with $x=\bx^{(J)}$ we have:
\begin{equation}\label{3fp_eq0}
    Z\,:=\,\sum_{j=1}^n \exp(\beta \bx^{(j)})\,=\,k\, \exp(\beta\,  x)+ (n-k)\, \exp\left(\beta\, \frac{1-kx}{n-k}\right).
\end{equation}
Then $\sm_\beta(\bx)=\bx$ if, and only the following holds:
\begin{equation}\label{3fp_eq1}
    \frac{\exp(\beta\, x)}{Z}\,=\, x.
\end{equation}
Indeed, given that (\ref{3fp_eq1}) holds, the fixed point identity holds for the components $j\in J$ of $\bx$. For the remaining indices $j\notin J$, we obtain with (\ref{3fp_eq1}) and (\ref{3fp_eq0})
\begin{eqnarray*}
    \frac{1-kx}{n-k} &\,=\,&(n-k)^{-1}\,  \left(1-\frac{k \,\exp(\beta\, x)}{k\, \exp(\beta\,  x)+ (n-k)\, \exp\left(\beta\, \frac{1-kx}{n-k}\right)}\right)\nonumber\\
    &\,=\,& (n-k)^{-1}\,  \frac{(n-k) \,\exp\left(\beta\, \frac{1-kx}{n-k}\right)}{k\, \exp(\beta\,  x)+ (n-k)\, \exp\left(\beta\, \frac{1-kx}{n-k}\right)}\\
    &\,=\,&\frac{\exp\left(\beta\, \frac{1-kx}{n-k}\right)}{Z}\nonumber,
\end{eqnarray*}
showing that the fixed point identity also holds for the components $j\notin J$ in $\bx$.
Next, \eqref{3fp_eq1} is equivalent to 
\begin{equation*}
    k\, \exp(\beta\, x) \,+\, (n-k)\,\exp\left(\beta\, \frac{1-kx}{n-k}\right)\,=\, \frac{\exp(\beta\, x)}{x}.
\end{equation*}
Dividing by $\exp(\beta\, x)$ and isolating the right summand on the left-hand side and taking logarithms gives
\begin{equation}\label{3fp_eq2}
    \log(n-k) + \beta\, \frac{1-nx}{n-k} = \log \left(\frac{1}{x}-k\right).
\end{equation}
To analyse (\ref{3fp_eq2}), let $g(x):= \log\left(\frac{1}{x} - k\right)$. Then 
$ g''(x) = (1 - 2kx)(x - kx^2)^{-2}$, from which it is seen that $g$ is strictly convex for $(0,\frac{1}{2k}] $ and $g$ is strictly concave on $ [\frac{1}{2k},\frac{1}{k}]$. As the left-hand side in (\ref{3fp_eq2}) is an affine function in $x$, the latter implies that (\ref{3fp_eq2}) has at most three solutions in $[0,1/k]$. To see that these three fixed points can be attained,  we assume $x\neq 1/n$ and isolate $\beta$ in \eqref{3fp_eq2}
 \begin{equation}\label{def_hk_01}
     \beta=\log\left(\frac{1/x-k}{n-k}\right) \, \frac{n-k}{1-nx}=:h(x)
 \end{equation}
 The previous assumption is justified by the fact that we seek fixed points other than $1/n\, \bi$. Let us recall that the latter is a fixed point for all $\beta$ and corresponds to $x=1/n$. By means of the Taylor expansion of the logarithm, we have
 \begin{equation}\label{tay_h}
     h(x)\,=\,\frac{1}{x}+\frac{1}{x}\, \sum_{j=1}^\infty\frac{(-1)^j}{(j+1)}  \left(\frac{1/x - n}{n-k } \right)^j
 \end{equation}
 in a neighborhood of $1/n$. It follows that $h$ is a positive analytic function on $(0,1/k)$ and the solution of \eqref{def_hk_01} determine the fixed points on $L_J$ other than $1/n\, \bi$. Considering again (\ref{def_hk_01}), it is seen that $h$ tends to $+\infty$ as $x$ approaches $0$ from the right and as $x$ approaches $1/k$ from the left. We can thus conclude that (\ref{def_hk_01}) has no solution for $0<\beta<\min\{h(x):x\in(0,1/k)\}$, and hence $1/n\,\bi$ is the only fixed point on $L_J$ for this case. On the other hand, the latter analysis on $h$ reveals that for sufficiently large $\beta$, we have two solutions for (\ref{def_hk}).

This then motivates to define, for some integer $n>3$, and $k\in\{1,\ldots,\lfloor n/2\rfloor\}$, we define
let \begin{equation}\label{def_hk}
    \begin{split}
        h_{n,k}(x)&\,:=\,\log\left(\frac{1/x-k}{n-k}\right) \, \frac{n-k}{1-nx},\quad x\in (0,1/k)\\
        m(n,k)&\,:=\,\min\{h_{n,k}(x):x\in (0,1/k)\}\\
        m(n,0)&\,:=\,0,\quad m(n,\lfloor n/2\rfloor+1):=\infty
    \end{split}
\end{equation}
As mentioned before, our notion deviates for \cite{tivno2009bifurcation} in that we formulated the results in terms of $\beta$ instead of $1/\beta$.

The following results now summarizes some results from \cite[Section 2]{tivno2009bifurcation} .
\begin{theorem}(Ti{\v{n}}o)\label{th_fix_scaledSoftmax_detailed}    
    Let $n$ be  some positive integer and $\beta>0$, such that 
    \begin{equation*}
        m(n,\nu)\,<\, \beta <\, m(n,\nu+1)
    \end{equation*}
    for some $\nu\in \{0,\ldots,\lfloor n/2\rfloor\}$. Then for positive integers $k\leq\nu$, there are exactly two $x(k)>y(k)$ in $(0,1/k)\setminus\{1/n\}$ that solve 
    \begin{equation*}
        \beta \,=\, h_{n,k}(x)
    \end{equation*}
    and that determined the fixed points of $\sm_\beta$:
    \begin{equation*}
        \fix(\sm_\beta) \, =\,\left\{\frac{1}{n}\bi\right\}\cup \ \bigcup_{k=1}^{\nu}\bigcup_{\substack{J\subset [n]\\ \vert J\vert =k}}\{l_J(x(k)),\, l_J(y(k))\}.
    \end{equation*}
\end{theorem}

\begin{figure}[htbp]
    \centering
    \begin{subfigure}{0.33\textwidth}
        \centering
        \includegraphics[width=\textwidth]{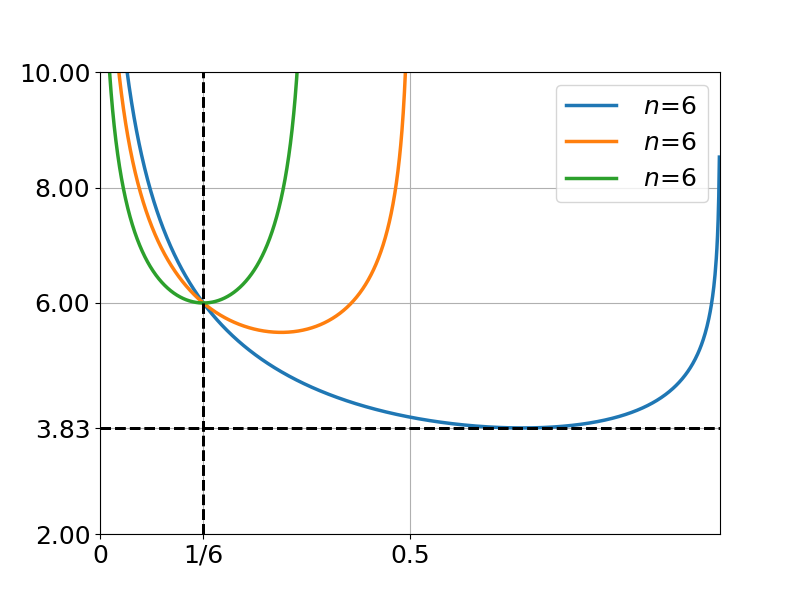}
        \caption{For $n=6$, $k=1,2,3$}
        \label{fig:hk_01}
    \end{subfigure}
    \hfill
    \begin{subfigure}{0.33\textwidth}
        \centering
        \includegraphics[width=\textwidth]{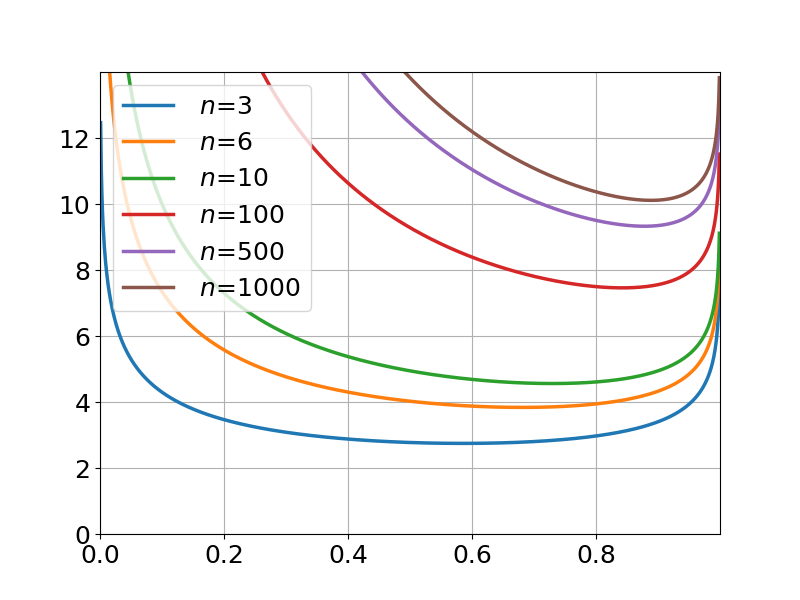}
        \caption{For $k=1$ and different $n$}
        \label{fig:hk_02}
    \end{subfigure}
    \begin{subfigure}{0.30\textwidth}
        \centering
        \begin{tabular}{|c|c|}
            \hline
            $n$ & $m(n,1)$ \\
            \hline
            3 & 2.746  \\
            6 & 3.836 \\
            10 & 4.559  \\
            100 & 7.459  \\
            500 & 9.326  \\
            1000 & 10.111\\
            \hline
        \end{tabular}
        \caption{Min. of $h_{n,1}$ over $(0,1)$}
        \label{tab:min_h_n}
    \end{subfigure}
    \caption{Graphs and minima of $h_{n,k}$ (cf. (\ref{def_hk}))}
    \label{fig:hk_all}
\end{figure}

A stability analysis on the fixed points of $\sm_\beta$ is given in \cite[Section 3]{tivno2009bifurcation}, wherein the eigenvalues and eigenvectors of Jacobians of fixed points are analysed in detail. 
Some findings of this section can be summarized as follwows in our terminology.

\begin{theorem}(Ti{\v{n}}o)\label{th_stablefix_scaledSoftmax_detailed}
For $\beta \neq n$, we have
        \begin{equation*}
            \afix(\sm_\beta) \, =\begin{cases}
                \,\left\{\frac{1}{n}\bi\right\} \quad &\text{for}\ \beta<m(n,1)\\[0.2cm]
                \left\{\frac{1}{n}\bi\right\}\cup \,\bigcup\limits_{j=1}^n\{l_{\{j\}}(x(k))\} \quad &\text{for}\ m(n,1)<\beta<n\\[0.2cm]
                \,\bigcup\limits_{j=1}^n\{l_{\{j\}}(x(k))\} \quad &\text{for}\ n<\beta.
            \end{cases}
    \end{equation*}

\end{theorem}

\end{document}